\documentclass[11pt,a4paper, reqno]{article}
\usepackage[utf8]{inputenc}

\usepackage[T1]{fontenc}

\usepackage{amsfonts, amsmath, amssymb, amsthm, bbm, enumerate, graphicx, mathtools, tikz, hyperref, enumerate,color}

\usepackage{subcaption} 

\usepackage[affil-it]{authblk}   
\usepackage[T1]{fontenc}
\usepackage[sort, numbers]{natbib}

\newcommand{\sss}{\scriptscriptstyle}

\newcommand{\e}{\ensuremath{\mathrm{e}}}

\long\def\/*#1*/{}
\newcommand\given[1][]{\:#1\middle|\:} 
\newcommand\prob[1]{\mathbbm{P}\left(#1\right)}  
\newcommand\expt[1]{\mathbbm{E}\left(#1\right)}  
\newcommand{\dto}{\ensuremath{\xrightarrow{d}}}  
\newcommand{\dif}{\ensuremath{d}} 

\newcommand\ind[1]{\ensuremath{\mathbbm{1}_{\left\{#1\right\}}}} 
\newcommand{\R}{\mathbbm{R}}                 
\newcommand{\Z}{\mathbbm{Z}}				

\newcommand{\1}{\ensuremath{\mathbbm{1}}}
\newcommand{\floor}[1]{\ensuremath{\left\lfloor #1 \right\rfloor}}
\newcommand{\bld}[1]{\ensuremath{\boldsymbol{#1}}}
\newcommand{\norm}[1]{\ensuremath{\left\|#1\right\|}}
\newcommand{\var}[1]{\ensuremath{\mathrm{Var}\left(#1\right)}}

\hypersetup{
  colorlinks,
  linkcolor=blue,
  citecolor=red,
  filecolor=blue,
  urlcolor=black,
  pdftitle={},
  pdfauthor={D. Mukherjee},
  pdfcreator={D. Mukherjee},
  pdfsubject={},
  pdfkeywords={}
}

\newtheorem{theorem}{Theorem}

\newtheorem{lemma}[theorem]{Lemma}
\newtheorem{proposition}[theorem]{Proposition}
\newtheorem{remark}{Remark}
\newtheorem{algo}{Algorithm}

\begin{document}

\title{Generalized random sequential adsorption on Erd\H{o}s-R\'enyi random graphs}

\author{Souvik Dhara}
\author{Johan S.H.~van Leeuwaarden}
\author{Debankur Mukherjee}
\affil{Department of Mathematics and Computer Science,\\
Eindhoven University of Technology, The Netherlands}

\renewcommand\Authands{ and }

\date{\today}
\maketitle

\begin{abstract}
We investigate Random Sequential Adsorption (RSA) on a random graph via the following greedy algorithm:
Order the $n$ vertices at random, and sequentially declare each vertex either {\it active} or {\it frozen}, depending on some local rule in terms of the state of the neighboring vertices.
The classical RSA rule declares a vertex active if none of its neighbors is, in which case the set of active nodes forms an independent set of the graph.  We generalize this nearest-neighbor blocking rule in three ways and apply it to the Erd\H{o}s-R\'enyi random graph. We consider these generalizations in the large-graph limit $n\to\infty$ and 
characterize the {\it jamming constant}, the limiting proportion of active vertices in the maximal greedy set.

\vspace{.2cm} \noindent {\bf Keywords:} random sequential adsorption; jamming limit; random graphs; parking problem; greedy independent set; frequency assignment
\end{abstract}

\section{Introduction}\label{sec:intro}

Random Sequential Adsorption (RSA) refers to a process in which particles appear sequentially at random positions in some space, and if accepted, remain at those positions forever. This strong form of irreversibility is often observed in dynamical interacting particle systems; see \cite{C14, RM06, P01, CAP07, LTDLV15} and the references therein for many applications across various fields of science. One example concerns particle systems with hard-core interaction, in which particles are accepted only when there is no particle already present in its direct neighborhood. In a continuum, the hard-core constraint says that particles should be separated by at least some fixed distance.

Certain versions of RSA are called \emph{parking problems} \cite{R58}, where cars of a certain length arrive at random positions on some interval (or on $\R$). Each car sticks to its location if it does not overlap with the other cars already present.
The fraction of space occupied when there is no more place for further cars is known as R\'enyi's parking constant.
RSA or parking problems were also studied on random trees \cite{DFK08, S09}, where the nodes of an infinite random tree are selected one by one, and are declared {\it active} if none of the neighboring nodes is active already, and become {\it frozen} otherwise.

We will study RSA on random graphs, where as in a tree, nodes either become active or frozen. We are interested in the
fraction of active nodes in the large-network limit when the number of nodes $n$ tends to infinity. We call this limiting fraction the {\it jamming constant}, and it can be interpreted as the counterpart of  R\'enyi's parking constant, but then for random graphs.
For classical RSA with nearest-neighbor blocking, the jamming constant corresponds to the normalized size of a greedy maximal independent set, where at each step one vertex is selected uniformly at random from the set of all vertices that have not been selected yet, and is included in the independent set if none of its neighbors are already included.
The size of the maximal greedy independent set of an Erd\H{o}s-R\'enyi random graph was first considered in \cite{M84}; see Remark~\ref{rem:K=1} below.
Recently, jamming constants for the Erd\H{o}s-R\'enyi random graph were studied in \cite{BJS15, SJK15}, and for random graphs with given degrees in \cite{BJM13, BJLM14, BJL15}. 
In \cite{BJLM14}, random graphs were used to model wireless networks, in which nodes (mobile devices) try to activate after random times, and can only become active if none of their neighbors is active (transmitting).
When the size of the wireless network becomes large and nodes try to activate after a short random time independently of each other, the jammed state with a maximal number of active nodes becomes the dominant state of the system. In \cite{SJK15}, random graphs with nearest-neighbor blocking were used to model a Rydberg gas with repelling atoms. In ultra-cold conditions, the repelling atoms with quantum interaction favor a jammed state, or frozen disorder, and in \cite{SJK15} it was shown that this jammed state could be captured in terms of the jamming limit of a random graph, with specific choices for the free parameters of the random graph to fit the experimental setting.

In this paper we consider three generalizations of RSA on the Erd\H{o}s-R\'enyi random graph.
The generalizations cover a wide variety of models, where the interaction between the particles is repellent, but not as stringent as nearest-neighbor blocking.
The first generalization is inspired by wireless networks.
Suppose that each active node causes one unit of noise to all its neighboring nodes. Further, a node is allowed to transmit (and hence to become active) unless it senses \emph{too much} noise, or causes too much noise to some already active node.
We assume that there is a \emph{threshold} value $K$ such that a node is allowed to become active only when the total noise experienced at that node is less than $K$, and the total noise that would be caused by the activation of this node to its neighboring active nodes, remains below $K$. We call this the \emph{Threshold model}. 
In the jammed state, all active  nodes  have fewer than $K$ active neighbors, and all frozen nodes would violate this condition when becoming active. 
This condition relaxes the strict hardcore constraint ($K=1$) and combined with RSA produces a greedy maximal $K$-independent set, defined as a subset of vertices $U$ in which each vertex has at most $K-1$ neighbors in $U$. 
The Threshold model was studied in \cite{M14, MBBD15} on two-dimensional grids in the context of distributed message spreading in wireless networks.

The second generalization considers a multi-frequency or multi-color version of classical RSA. There are $K$ different frequencies available.
A node can only receive a `higher' frequency than any of its already active neighbors.
Otherwise the node gets frozen. As in the Threshold model, the case $K=1$ reduces to the classical hard-core constraint. 
But for $K\geq 2$, this multi-frequency version gives different jammed states, and is also known as RSA with screening or the \emph{Tetris model} \cite{FK10}. 
In the Tetris model, particles sequentially drop from the sky, on random locations (nodes in case of graphs), and stick to a height that is one unit higher than the heights of the particles that occupy neighboring locations.
This model has been studied in the context of ballistic particle deposition \cite{M83}, where particles dropping vertically onto  a surface stick to a location when they hit either a previously deposited particle or the surface.

The third generalization concerns random \emph{Sequential Frequency Assignment Process} (SFAP) \cite{FD07,FK09}. 
As in the Tetris model, there are $K$ different frequencies, and a node cannot use a frequency at which one of its neighbors is already transmitting.
But this time, a new node selects the lowest available frequency. If there is no further frequency available  (i.e.~all the $K$ different frequencies are taken by its neighbors), the node becomes frozen.  The SFAP model can be used as a simple and easy-to-implement algorithm for determining  interference-free frequency assignment in  radio communications  regulatory  services \cite{FD07}.

The paper is structured as follows. Section~\ref{sec:results} describes the models in detail and presents the main results. We quantify how the jamming constant depends on the value of $K$ and the edge density of the graph. Section~\ref{sec:proofs} gives the proofs of all the results and Section~\ref{sec:fu} describes some further research directions.

\subsection*{Notation and terminology}  We denote an Erd\H{o}s-R\'enyi random graph on the vertex set $[n]=\{1,2,\dots,n\}$  by $G(n,p_n)$, where for any $u\neq v$, $(u,v)$ is an edge of $G(n,p_n)$ with probability $p_n=c/n$ for some $c>0$, independently for all distinct $(u,v)$-pairs.
We often mean by $G(n,p_n)$ the distribution of all possible configurations of the Erd\H{o}s-R\'enyi random graph with parameters $n$ and $p_n$, and we sometimes omit sub-/superscript $n$ when it is clear from the context. 
The symbol $\1_A$ denotes the indicator random variable corresponding to the set $A$. An empty sum and an empty product is always taken to be zero and one respectively.
We use calligraphic letters such as $\mathcal{A},$ $\mathcal{I}$, to denote sets, and the corresponding normal fonts such as $A$, $I$, to denote their cardinality.
Also, for discrete functions $f:\{0,1,\dots\}\mapsto \R$ and $x>0$, $f(x)$ should be understood as $f(\floor{x})$. The boldfaced notations such as $\bld{x}$, $\bld{\delta}$ are reserved to denote vectors, and $\norm{\cdot}$ denotes the sup-norm on the Euclidean space. The convergence in distribution statements for processes are to be understood as uniform convergence over compact sets.

\section{Main results}\label{sec:results}
We now present  the three models in three separate sections. For each model, we describe an algorithm that lets the graph grow and simultaneously applies RSA. Asymptotic analysis of the algorithms in the large-graph limit $n\to\infty$ then leads to characterizations of the jamming constants. 

\subsection{Threshold model}
For any graph $G$ with vertex set $V$, let $d_{\max}(G)$ denote the maximum degree, and denote the subgraph induced by $U\subset V$ as $G_{\sss U}$. Define the configuration space as
\begin{equation}
\Omega_K(G)=\{U\subset V:d_{\max}(G_{\sss U})<K\}.
\end{equation}
We call any member of  $\Omega_K(G)$ a \emph{$K$-independent set} of $G$.
Now consider the following process on $G(n,p_n)$: Let $\mathcal{I}(t)$ denote the set of active nodes at time $t$, and $\mathcal{I}(0):=\varnothing$. Given $\mathcal{I}(t)$, at step $t+1$, one vertex $v$ is selected uniformly at random from the set of all vertices which have not been selected already, and if $d_{\max}(G_{\sss \mathcal{I}(t)\cup \{v\}})<K$, then set $\mathcal{I}(t+1)=\mathcal{I}(t)\cup\{v\}$. Otherwise, set $\mathcal{I}(t+1)=\mathcal{I}(t).$
Note that, given the graph $G(n,p_n)$, $\mathcal{I}(t)$ is a random element from $\Omega_K(G(n,p_n))$ for each $t$, and after $n$ steps we get a \emph{maximal} greedy $K$-independent set. 
We are interested in the jamming fraction $I(n)/n$ as $n$ grows large, and we call the limiting value, if it exists, the jamming constant.

To analyze the jamming constant for the Threshold model, we introduce an exploration algorithm that generates both the random graph and the greedy $K$-independent set simultaneously. The algorithm thus outputs a maximal $K$-independent set equal in distribution to $\mathcal{I}(n)$.

\begin{algo}[{Threshold exploration}]\label{algo-th}
\normalfont At time $t$, we keep track of the sets $\mathcal{A}_k(t)$ of active vertices that have precisely $k$ active neighbors, for $0\leq k\leq K-1$, the set $\mathcal{B}(t)$ of frozen vertices, and the set $\mathcal{U}(t)$ of unexplored vertices.
Initialize by setting $\mathcal{A}_k(0)=\mathcal{B}(0)=\varnothing$ for $0\leq k\leq K-1$, and $\mathcal{U}(0)=V$. Define $\mathcal{A}(t):=\bigcup_{k}\mathcal{A}_k(t)$.
At time $t+1$, if $\mathcal{U}(t)$ is nonempty, we select a vertex $v$ from $\mathcal{U}(t)$ uniformly at random and try to pair it with the vertices of $\mathcal{A}(t)\cup\mathcal{B}(t)$, mutually independently, with probability $p_n$.
Suppose the set of all vertices in $\mathcal{A}(t)$ to which the vertex $v$ is paired is given by $\{v_1,\dots,v_r\}$ for some $r\geq 0$, where for all $i\leq r$, $v_i\in \mathcal{A}_{k_i}(t)$ for some $0\leq k_i\leq K-1$. Then:
\begin{itemize}
\item If $r<K$ and $v_i\notin \mathcal{A}_{K-1}(t)$ for all $i\leq r$ (i.e.~$\max_i k_i<K-1$), then put $v$ in $\mathcal{A}_r(t)$ and move each $v_i$ from $\mathcal{A}_{k_i}(t)$ to $\mathcal{A}_{k_i+1}(t)$.
More precisely, set
\begin{equation*}
\begin{aligned}[t]
& \mathcal{A}_r(t+1)=\mathcal{A}_r(t)\cup \{v\},\\
& \mathcal{A}_{k_i+1}(t+1)= \mathcal{A}_{k_i}(t)\cup \{v_i\},\\
& \mathcal{U}(t+1)=\mathcal{U}(t)\setminus \{v\}.
  \end{aligned}
  \hspace*{1cm}
  \begin{aligned}[t]
&  \mathcal{A}_{k_i}(t+1)=\mathcal{A}_{k_i}(t)\setminus \{v_1,\dots,v_r\},\\
&  \mathcal{B}(t+1)=\mathcal{B}(t),\\
  \end{aligned}
\end{equation*}

\item Otherwise, if $r\geq K$ or $\mathcal{A}_{K-1}(t)\cap \{v_1,\dots,v_r\} \neq \varnothing$, declare $v$ to be blocked, i.e.~$\mathcal{B}(t+1)= \mathcal{B}(t)\cup \{v\}$, $\mathcal{A}_{k}(t+1)= \mathcal{A}_{k}(t)$ for all $0\leq k\leq K-1$ and $\mathcal{U}(t+1)=\mathcal{U}(t)\setminus \{v\}$.
\end{itemize}
\end{algo}
The algorithm terminates at $t=n$ and produces as output the set $\mathcal{A}(n)$ and a graph $\mathcal{G}(n)$. The following result guarantees that we can use Algorithm~\ref{algo-th} for analyzing the Threshold model:
\begin{proposition}\label{prop:thres}
The joint distribution of $(\mathcal{G}(n),\mathcal{A}(n))$ is identical to the joint distribution of $(G(n,p_n),\mathcal{I}(n))$.
\end{proposition}
Observe that $|\mathcal{U}(t)|=n-t$.
Our goal is to find the jamming constant, i.e. the asymptotic value of $A(n)/n$. For that purpose, define $\alpha^n_k(t):=A_k(\floor{nt})/n$, and the vector $\bld{\alpha}^n(t)=(\alpha^n_0(t),\ldots,\alpha^n_{K-1}(t))$ for $t\in [0,1]$.
 We can now state the main result for the Threshold model.
\begin{theorem}[{Threshold jamming limit}]\label{th:threshold}
The process $\{\bld{\alpha}^n(t)\}_{0\leq t\leq1}$ on  $G(n,p_n)$, with $p_n=c/n$, converges in distribution  to the deterministic process $\{\bld{\alpha}(t)\}_{0\leq t\leq 1}$ that can be described as the unique solution of the  integral recursion equation
\begin{equation}\label{eq:alpha-threshold}
 \alpha_k(t)=\int_0^t\delta_k(\bld{\alpha}(s))\dif s,
\end{equation} where
\begin{equation}\label{def:delta-th}
 \delta_k(\bld{\alpha})=
 \begin{cases}
  -c\alpha_{0}\e^{-c\alpha_{\sss \leq K-1}}\sum_{r=0}^{K-2}c^r\alpha_{\sss \leq K-2}^r/r!+\e^{-c\alpha_{\sss \leq K-1}},  &k=0,\\
  c(\alpha_{k-1}-\alpha_k)\e^{-c\alpha_{\sss \leq K-1}}\sum_{r=0}^{K-2}c^r\alpha_{\sss \leq K-2}^r/r! & 1\leq k\leq K-2,\\
 \hspace*{2cm} +\e^{-c\alpha_{\sss \leq K-1}}c^k\alpha_{\sss \leq K-2}^k/k!,&\\
 c\alpha_{K-2}\e^{-c\alpha_{\sss \leq  K-1}}\sum_{r=0}^{K-2}c^r\alpha_{\sss \leq  K-2}^r/r! & k=K-1,\\
 \hspace*{2cm} +\e^{-c\alpha_{\sss \leq K-1}}c^{K-1}\alpha_{\sss \leq K-2}^{K-1}/(K-1)!,&
 \end{cases}
\end{equation} with $\alpha_{\sss \leq k}=\alpha_0+\dots+\alpha_k$.
Consequently, as $n\to\infty$, the jamming fraction converges in distribution to a constant, i.e.,
\begin{equation}
\frac{I(n)}{n}\dto \sum_{k=0}^{K-1}\alpha_k(1).
\end{equation}
\end{theorem}

Figure~\ref{fig-th} displays some numerical values for the fraction of active nodes given by $ \sum_{k=0}^{K-1}\alpha_k(1)$, as a function of 
the average degree $c$ and the threshold $K$. As expected, an increased threshold $K$ results in a larger fraction.
Figure~\ref{fig-th} also shows prelimit values of this fraction for a finite network of $n=1000$ nodes. These values are obtained by simulation, where for each value of $c$ we show the result of one run only.
 This leads to the rougher curves that  closely follow the smooth deterministic curves of the jamming constants. If we had plotted the average values of multiple simulation runs, 100 say, this average simulated curve would be virtually indistinguishable from the smooth curve. This not only confirms that our limiting result is correct, but it also indicates that the limiting results serve as good approximations for finite-sized networks.  We have drawn similar conclusions based on extensive simulations for all the jamming constants presented in this section. 

\begin{figure}
\begin{center}
\includegraphics[scale=.45]{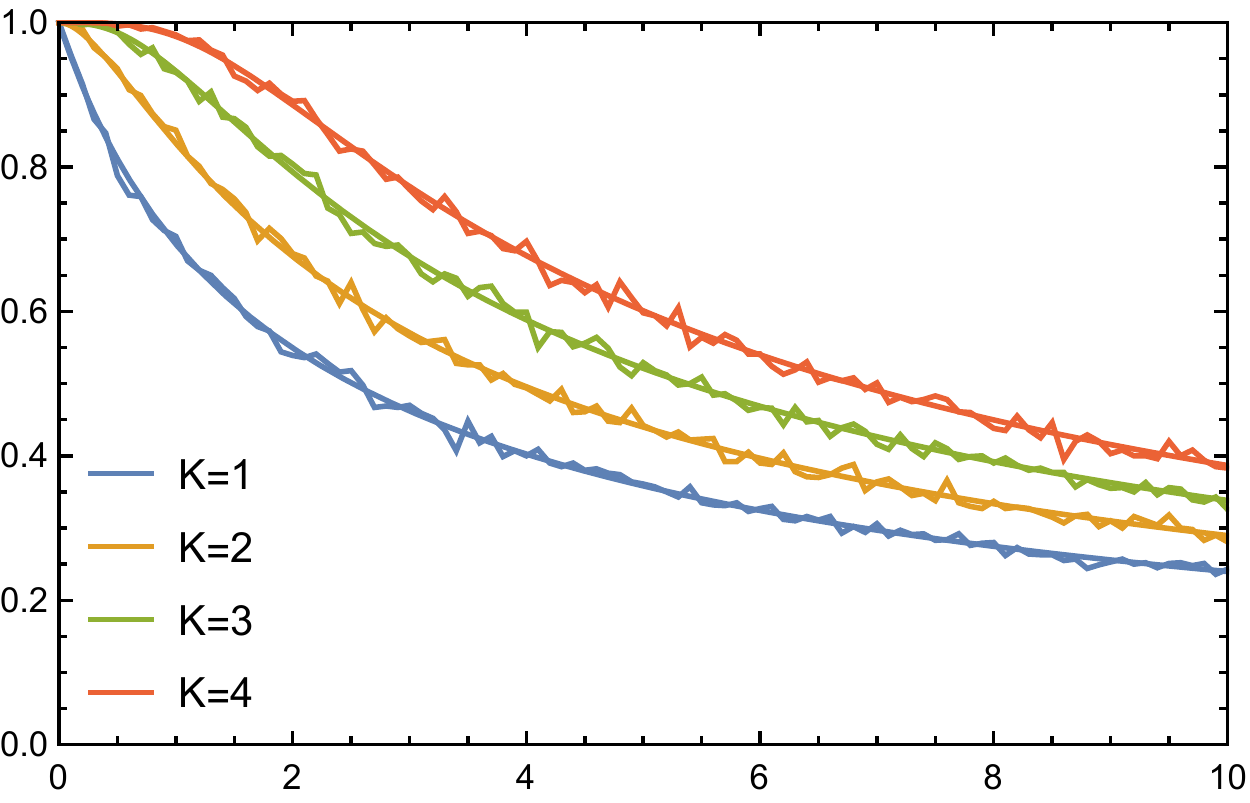}
\end{center}
\caption{Fraction of active nodes as a function of $c$, for $0\leq c\leq 10$ and several $K$-values. The smooth lines display $\sum_{k=0}^{K-1}\alpha_k(1)$ and the rough lines follow from simulation of a network with $n=1000$ nodes.}
\label{fig-th}
\end{figure}

\begin{remark}\label{rem:K=1}
\normalfont
It can be checked that Theorem~\ref{th:threshold} gives the known jamming constant for $K=1$. In this case, \eqref{eq:alpha-threshold} reduces to 
 \begin{equation}
 \alpha_0(t)=\int_0^t\e^{-c\alpha_0(s)}\dif s
 \end{equation}
 with $\alpha_0(0)=0$. Thus the value of the jamming constant becomes $\alpha_0(1)=c^{-1}\log(1+c)$, which agrees with the known value \cite[Theorem 2.2 (ii)]{M84}.
\end{remark}

\begin{remark}
\normalfont
Theorem \ref{th:threshold} can be understood intuitively as follows.
Observe that when a vertex $v$ is selected from $\mathcal{U}$, it will only be added to $\mathcal{A}_k$ if it is not connected to $\mathcal{A}_{K-1}$, and it has precisely $k\leq K-1$ connections to the rest of $\mathcal{A}$. Further, if the selected vertex $v$ becomes active, then all the vertices in $\mathcal{A}_j$, to which $v$ gets connected, are moved to $\mathcal{A}_{j+1}$, $0\leq j<K-1$. 
The number of connections to $\mathcal{A}_{k}$, in this case, is Bin$(A_{k},p_n)$ and that to $\bigcup_{i=0}^{K-2}\mathcal{A}_i$ is Bin$(A_{\sss\leq K-2},p_n)$, and we have the additional restriction that the latter is less than or equal to $K-1$. The expectation of Bin$(A_{k},p_n)$ restricted to  Bin$(A_{\sss\leq K-2},p_n)\leq K-1$ is given by $A_kp_n$ times another binomial probability (see Lemma~\ref{lem:bin-condn-expt}). This explains the first terms on the right side of \eqref{def:delta-th}.
Finally, taking into account the probability of acceptance to $\mathcal{A}_k$ gives rise to the second terms on the right side of \eqref{def:delta-th}. 
\end{remark}
\begin{remark}\normalfont Algorithm~\ref{algo-th} is different in spirit than the  exploration algorithms in the recent works \cite{SJK15,BJL15}.
The standard greedy algorithms in \cite{SJK15,BJL15} work as follows: Given a graph $G$ with $n$ vertices, include the vertices in an independent set $I$ consecutively, and at each step, one vertex is chosen randomly from those not already in
the set, nor adjacent to a vertex in the set. This algorithm must find all vertices adjacent to
$I$ as each new vertex is added to $I$, which requires probing
all the edges adjacent to vertices in~$I$.
However, since in the Threshold model with $K\geq 2$ an active node does not obstruct its neighbors from activation per se, we need to keep track of the nodes that are neither active nor blocked. We deal with this additional complexity by simply observing that the activation of a node is determined by exploring the connections with the previously active vertices only. Therefore, Algorithm~\ref{algo-th} only describes the connections between the new (and potentially active) vertex and the already active vertices (and the frozen vertices in order to complete the construction of the graph). Since the graph is built one vertex at a time, the jamming state is achieved precisely at time $t=n$, and not at some random time in between~$1$ and~$n$ as in  \cite{SJK15,BJL15}.
\end{remark}
\begin{remark}\normalfont
For the other two RSA generalizations discussed below we will use a similar algorithmic approach, building and exploring the graph  one vertex at a time. These algorithms form a crucial ingredient of this paper because they make the RSA processes amenable to analysis.
\end{remark}
\subsection{Tetris model}
In the Tetris model, particles are sequentially deposited on the vertices of a graph. For a vertex $v$, the incoming particle sticks at some height $h_v\in [K]=\{1,2,\ldots,K\}$ determined  by the following rules:
At time $t=0$, initialize by setting $h_v(0)=0$ for all $v\in V$. Given $\{h_v(t):v\in V\}$, at time $t+1$, one vertex $u$ is selected uniformly at random  from the set of all vertices that have not been selected yet. Set $h_u(t+1)=\max\{h_w(t):w\in V_u\}+1$ if $\max\{h_w(t):w\in V_u\}<K$, where $V_u$ is the set of neighboring vertices of $u$, and  set $h_u(t+1)=0$ otherwise.
Observe that the height of a vertex can change only once, and in the jammed state no further vertex at zero height can achieve non-zero height. 
Note that $K$ now has a different interpretation than in the Threshold model. In the Tetris model, the number of possible states on any vertex ranges from $0$ and $K$, whereas in the Threshold model the vertices have only two possible states (active/frozen), and $K$ determines ``the flexibility" in the acceptance criterion.
We are interested in the height distribution in the jammed state. Define $\mathcal{N}_i(t):=\{v:h_v(t)=i\}$ and $N_i(t)=|\mathcal{N}_i(t)|$.
We study the scaled version of the vector $(N_1(n),\ldots, N_K(n))$ and refer to $N_i(n)/n$ as the \emph{jamming density} of height $i$.

Again, we assume that the underlying interference graph is an Erd\H{o}s-R\'enyi random graph on $n$ vertices with independent edge probabilities $p_n=c/n$, and we use a suitable exploration algorithm that generates the random graph and the height distribution simultaneously.
\begin{algo}[{Tetris exploration}]\label{algo-tetris}
{\normalfont
At time $t$, we keep track of the set  $\mathcal{A}_k(t)$ of vertices at height $k$, for $0\leq k\leq K$, and the set $\mathcal{U}(t)$ of unexplored vertices. Initialize by putting $\mathcal{A}_{k}=\varnothing$ for $0\leq k\leq K$ and $\mathcal{U}(0)=V$.
Define $\mathcal{A}(t):=\bigcup_{k}\mathcal{A}_k(t)$. At time $t+1$, if $\mathcal{U}(t)$ is nonempty, we select a vertex $v$ from $\mathcal{U}(t)$ uniformly at random and try to pair it with the vertices of $\mathcal{A}(t)$, independently, with probability $p_n$.
 Suppose that the set of all vertices in $\mathcal{A}(t)$ to which the vertex $v$ is paired, is given by $\{v_1,\dots,v_r\}$ for some $r\geq 0$, where each $v_i\in \mathcal{A}_{k_i}(t)$ for some $0\leq k_i\leq K$. Then:
\begin{itemize}
\item When $\max_{i\in[r]} k_i\leq K-1$, set $h_v(t+1)=\max_{i\in[r]} k_i+1$, and $h_u(t+1)=h_u(t)$ for all $u\neq v$.
\item Otherwise $h_u(t+1)=h_u(t)$ for all $u\in V$.
\end{itemize}
The algorithm terminates at time $t=n$, when $\mathcal{U}(t)$ becomes empty, and outputs the  vector $(\mathcal{A}_1(n),\ldots,\mathcal{A}_K(n))$ and a graph $\mathcal{G}(n)$.
}
\end{algo}

\begin{proposition}\label{prop:tetris}
The joint distribution of  $(\mathcal{G}(n),\mathcal{A}_1(n),\ldots,\mathcal{A}_K(n))$ is identical to that of $(G(n,p_n), \mathcal{N}_1(n),\ldots, \mathcal{N}_K(n))$.
\end{proposition}
Due to Proposition \ref{prop:tetris} the desired height distribution can be obtained from  the scaled output produced by Algorithm \ref{algo-tetris}.
Define $\alpha_k^n(t)=A_k(nt)/n$ as before. Here is then the main result for the Tetris model:
\begin{theorem}[Tetris jamming limit]\label{th:tetris}
The process $\{\bld{\alpha}^n(t)\}_{0\leq t\leq1}$ on the graph $G(n,c/n)$ converges in distribution  to the deterministic process $\{\bld{\alpha}(t)\}_{0\leq t\leq 1}$ that can be described as the unique solution of the integral recursion equation
\begin{equation}\label{eq:alpha-tetris}
 \alpha_k(t)=\int_0^t\delta_k(\bld{\alpha}(s))\dif s,
\end{equation} where
\begin{equation}\label{def:delta-tetris}
 \delta_k(\bld{\alpha})=\begin{cases} \big( 1-\e^{-c\alpha_{k-1}}\big)\e^{-c(\alpha_k+\dots+\alpha_{\sss K})}, \quad &\mathrm{for }\ k\geq 2,\\
 \e^{-c(\alpha_1+\dots+\alpha_K)}, &\mathrm{for }\ k=1.
 \end{cases}
\end{equation}
Consequently, the jamming density of height $k$ converges in distribution to a constant, i.e.,~for all $1\leq k\leq K$,
\begin{equation}
\frac{N_k(n)}{n}\dto \alpha_k(1), \quad {\rm as} \  n\to\infty.
\end{equation}
\end{theorem}

Figure~\ref{fig:tetris-layer} shows the jamming densities of the different heights for $K=2,3,4$ and increasing average degree $c$. 
\begin{figure}
\begin{subfigure}{.5\textwidth}
  \centering
  \includegraphics[scale=.45]{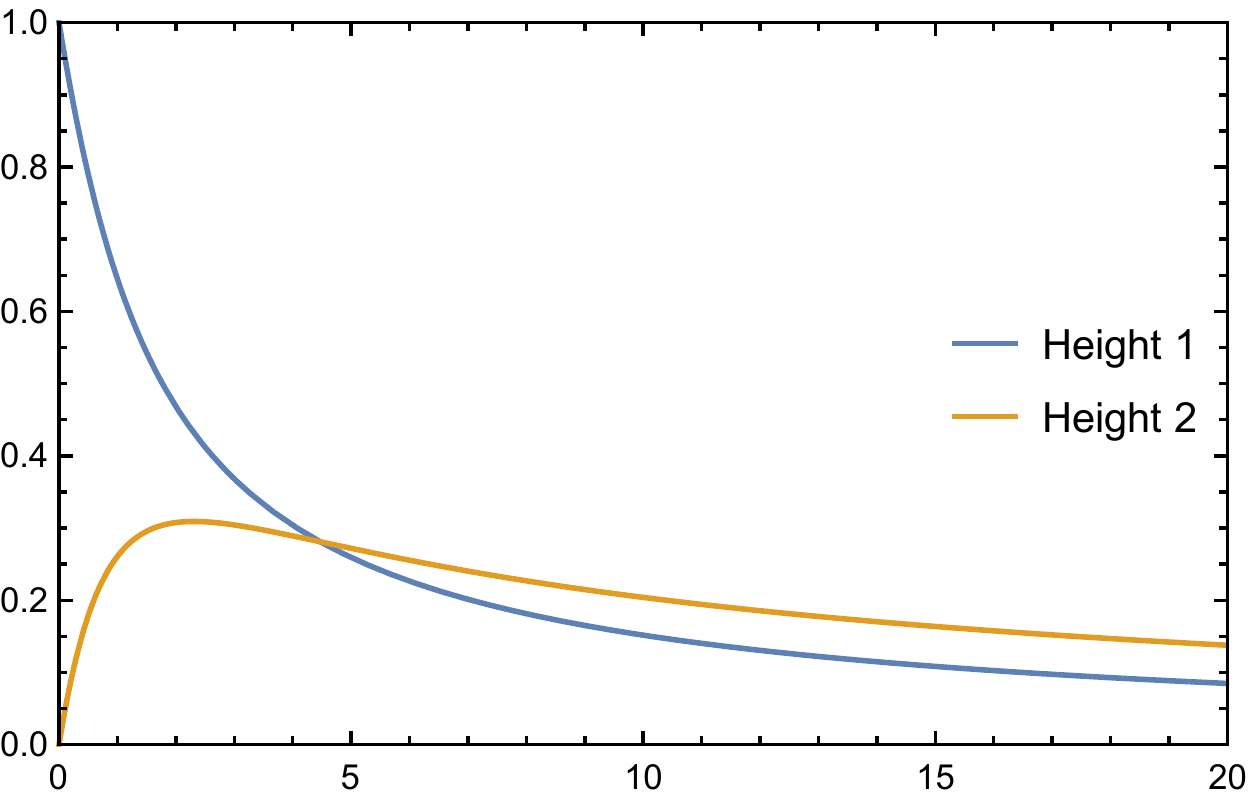}
  \caption{$K=2$}
  \label{fig:sfig1}
\end{subfigure}%
\begin{subfigure}{.5\textwidth}
  \centering
  \includegraphics[scale=.45]{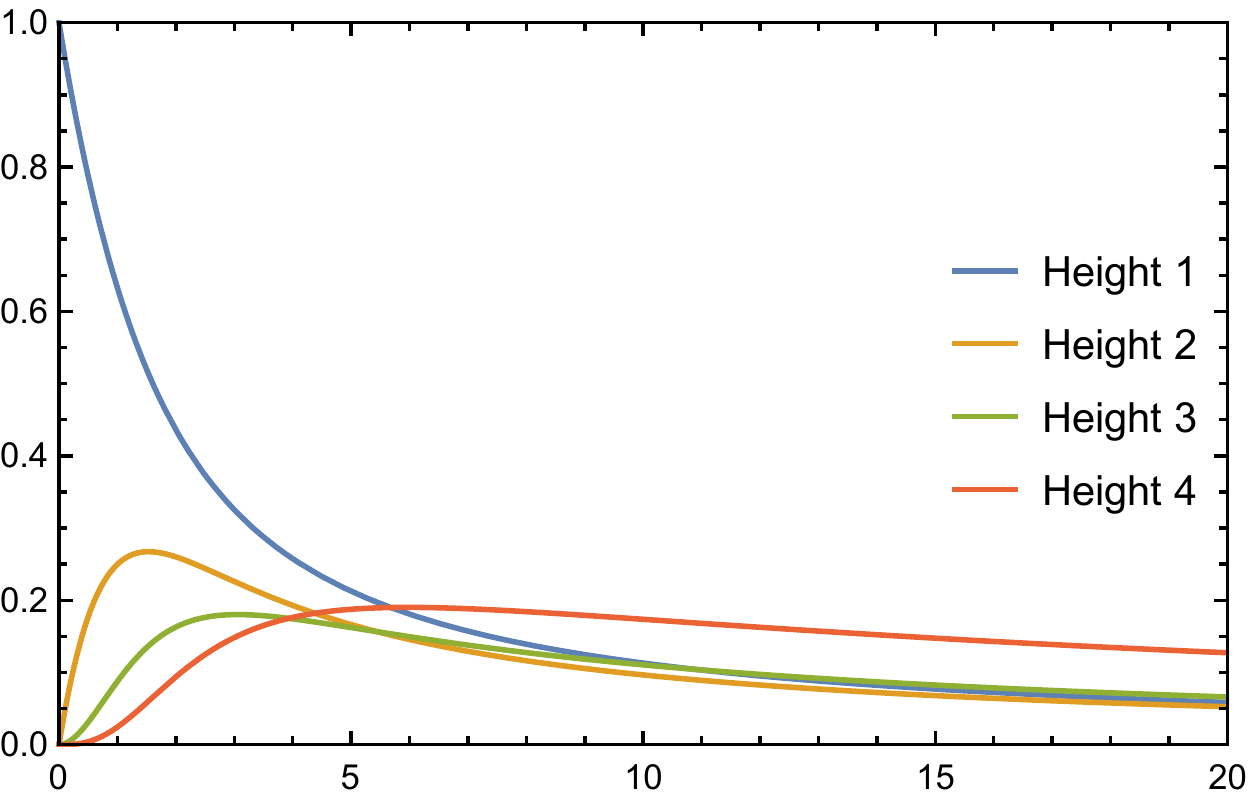}
  \caption{$K=4$}
  \label{fig:sfig3}
\end{subfigure}
\caption{Jamming densities of the different heights in the Tetris model as a function of $c$ for $0\leq c\leq 20$}
\label{fig:tetris-layer}
\end{figure}
Observe that in general the jamming heights do not obey a natural order. 
For $K=2$, for instance, the order of active nodes at heights one and two changes around $c\approx 4.4707$. Similar regime switches occur for large $K$-values as well. In general, for relatively sparse graphs with small $c$, the density of active nodes can be seen to decrease with the height, possibly due to the presence of many small subgraphs (like isolated vertices or pair of vertices). But as $c$ increases, the screening effect becomes prominent, and the densities increase with the height. 
Related phenomena have been observed for parking on $\Z$ \cite{FK09}, and on a random tree \cite{FK10}. 
 However, the models considered in \cite{FK09,FK10} are   different from the ones considered here in the sense that the heights \cite{FK09,FK10} are unbounded (there are no frozen sites as in this paper).
 Furthermore, Fig.~\ref{fig-tetr} displays the fraction of active nodes as a function of the average degree. 
 Notice here also that the jamming constant is increasing with $K$, as expected.


Theorem \ref{th:tetris} also gives the jamming constant $\alpha_1(1)+\cdots+\alpha_K(1)$ for the limiting fraction of active, or non-zero nodes.
For $K=1$ this corresponds to the fraction of nodes contained in the greedy maximal independent set, and as expected, relaxing the hard constraints by introducing more than one height ($K\geq 2$) considerably increases this fraction.
 
\begin{figure}
\begin{center}
\includegraphics[scale=.45]{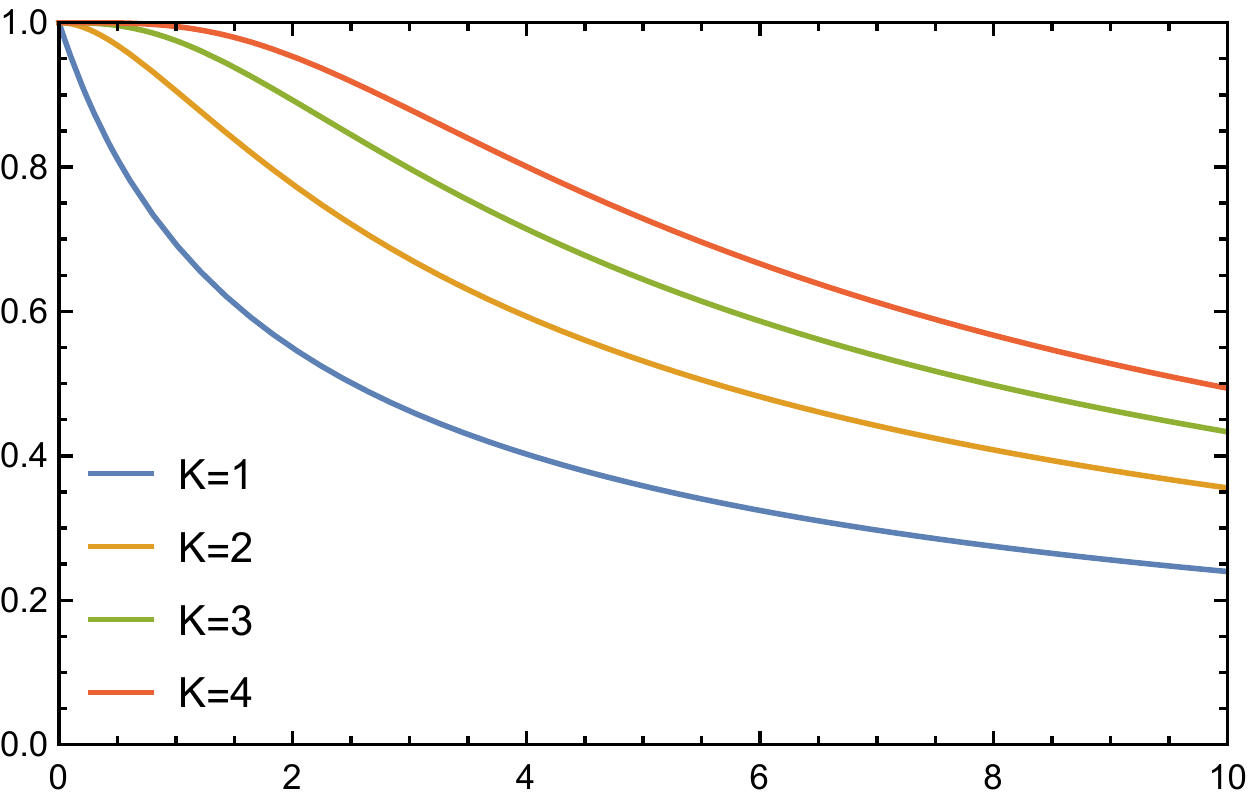}
\end{center}
\caption{Fraction of active nodes in the Tetris model as a function of $c$ for $0\leq c\leq 10$}
\label{fig-tetr}
\end{figure}

\begin{remark}
\normalfont
As in the Threshold model, it can be observed that the number of connections to the set $\mathcal{A}_i$ will be distributed approximately as Poi$(c\alpha_i)$. Now, at any step, a selected vertex $v$ is added to $\mathcal{A}_i$ if and only if it has a connection to at least some vertex in $\mathcal{A}_{i-1}$, and has no connections with $\bigcup_{j=i}^K\mathcal{A}_{j}$. The probability of this event can be recognized in the function $\delta_k$.
\end{remark}

\subsection{SFAP model}
The Sequential Frequency Assignment Process (SFAP) works as follows: Each node can take one of $K$ different frequencies indexed by $\{1,2,\ldots,K\}$, and neigboring nodes are not allowed to have identical frequencies, because this would cause a \emph{conflict}.
One can see that if the underlying graph $G$ is not $K$-colorable, then a conflict-free frequency assignment to all the vertices is ruled out. The converse is also true: If there is a feasible $K$-coloring for the graph $G$, then there exists a conflict-free frequency assignment.
Determining the optimal frequency assignment, in the sense of the maximum number of nodes getting at least some frequency for transmission, can be seen to be NP-hard in general (notice that $K=1$ gives the maximum independent set problem). This creates the need for distributed algorithms that generate a maximal (not necessarily maximum) conflict-free frequency assignment. The SFAP model provides such a distributed scheme~\cite{FD07}. 
As in the Threshold model and Tetris model, the vertices are selected one at a time, uniformly at random amongst those that have not yet been selected. A selected vertex probes its neigbors and selects the lowest available frequency. When all $K$ frequencies are already taken by its neighbors, the vertex gets no frequency and is called {\it frozen}.

Denote by $f_v(t)$ the frequency of the vertex $v$, and by $\mathcal{N}_i(t)$ the set of all vertices using frequency $i$ at time step $t$. 
As before, we are interested in the jamming density $N_i(n)/n$ of each frequency $1\leq i\leq K$. Again we consider the Erd\H{o}s-R\'enyi random graph, and the exploration algorithm is quite similar to that of the Tetris model, except for different local rules for determining the frequencies.

\begin{algo}[{SFAP exploration}]\label{algo-sfap}
{\normalfont
At time $t$, we keep track of the set  $\mathcal{A}_k(t)$ of vertices currently using frequency $k$, for $1\leq k\leq K$, the set $\mathcal{A}_0(t)$ of vertices that have been selected before time $t$, but did not receive a frequency (frozen), and the set $\mathcal{U}(t)$ of unexplored vertices. Initialize by setting $\mathcal{A}_{k}=\varnothing$ for $0\leq k\leq K$ and $\mathcal{U}(0)=V$.
Define $\mathcal{A}(t):=\bigcup_{k}\mathcal{A}_k(t)$.
At time $t+1$, if $\mathcal{U}(t)$ is nonempty, we select a vertex $v$ from $\mathcal{U}(t)$ uniformly at random and try to pair it with all vertices in $\mathcal{A}(t)$, independently with probability $p_n$.
Suppose that the set of all vertices in $\mathcal{A}(t)\setminus\mathcal{A}_0(t)$ to which the vertex $v$ is paired is given by $\{v_1,\dots,v_r\}$ for some $r\geq 0$, where each $v_i\in \mathcal{A}_{k_i}(t)$ for some $1\leq k_i\leq K$. Then:
\begin{itemize}
\item If the set $\mathcal{F}_v(t):=\{1,\dots,K\}\setminus \{k_i:1\leq i\leq r\}$ of non-conflicting frequencies is nonempty, then assign the vertex $v$ the frequency $f_v(t+1)=\min \mathcal{F}_v(t)$, and $f_u(t+1)=f_u(t)$ for all $u\in \mathcal{A}(t)$.
\item Otherwise set $f_v(t+1)=0$, and $f_u(t+1)=f_u(t)$ for all $u\in \mathcal{A}(t)$.
\end{itemize}
The algorithm terminates at time $t=n$ and outputs  $(\mathcal{A}_1(n),\ldots,\mathcal{A}_K(n))$ and a graph $\mathcal{G}(n)$. Again, we can show that this algorithm produces the right distribution.
}
\end{algo}

\begin{proposition}\label{prop:sfap}
The joint distribution of  $(\mathcal{G}(n),\mathcal{A}_1(n),\ldots,\mathcal{A}_K(n))$ is identical to that of $(G(n,p_n)$, $\mathcal{N}_1(n),\ldots, \mathcal{N}_K(n))$.
\end{proposition}
Again, define $\alpha_k^n(t)=A_k(nt)/n$.
\begin{theorem}[{SFAP jamming limit}]\label{th:sfap}
The process $\{\bld{\alpha}^n(t)\}_{0\leq t\leq1}$ converges in distribution  to the process $\{\bld{\alpha}(t)\}_{0\leq t\leq 1}$ that can be described as the unique solution to the  deterministic integral recursion equation
\begin{equation}\label{eq:alpha-SFAP}
 \alpha_k(t)=\int_0^t\delta_k(\bld{\alpha}(s))\dif s,
\end{equation} where, for all $1\leq k\leq K$,
\begin{equation}\label{eq:delta-sfap}
 \delta_k(\bld{\alpha})=\e^{-c\alpha_k}\prod_{r=1}^{k-1}\big(1-\e^{-c\alpha_r}\big).
\end{equation}
 Consequently, the jamming density at height $k$ converges in probability to a constant, i.e.~for all $1\leq k\leq K$,
\begin{equation}
\frac{N_k(n)}{n}\dto \alpha_k(1), \quad {\rm as} \ n\to\infty.
\end{equation}
\end{theorem}
It is straightforward to check that the system of equations in \eqref{eq:alpha-SFAP} has the solution
\begin{equation}\label{eq:tetris}
\begin{split}
\alpha_1(t)&=\frac{1}{c}\log(1+ct),\\
\alpha_i(t)&=\frac{1}{c}\log(\e^{c\alpha_{i-1}(t)}-c\alpha_{i-1}(t)),\quad\mbox{for }i\geq 2.
\end{split}
\end{equation}

\begin{figure}
\begin{subfigure}{.5\textwidth}
  \centering
  \includegraphics[scale=.45]{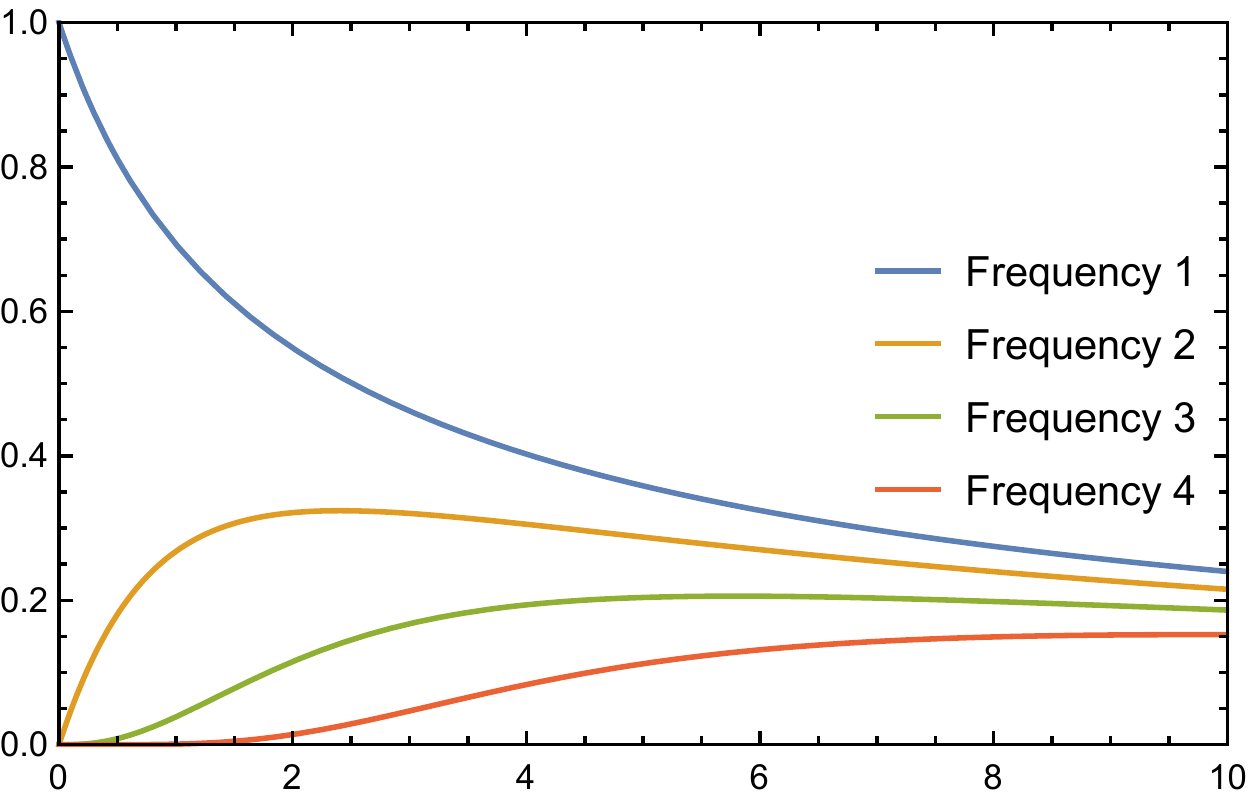}
  \caption{Densities at different frequencies as a function of $c$}
  \label{fig:SFAP}
\end{subfigure}%
\begin{subfigure}{.5\textwidth}
  \centering
  \includegraphics[scale=.45]{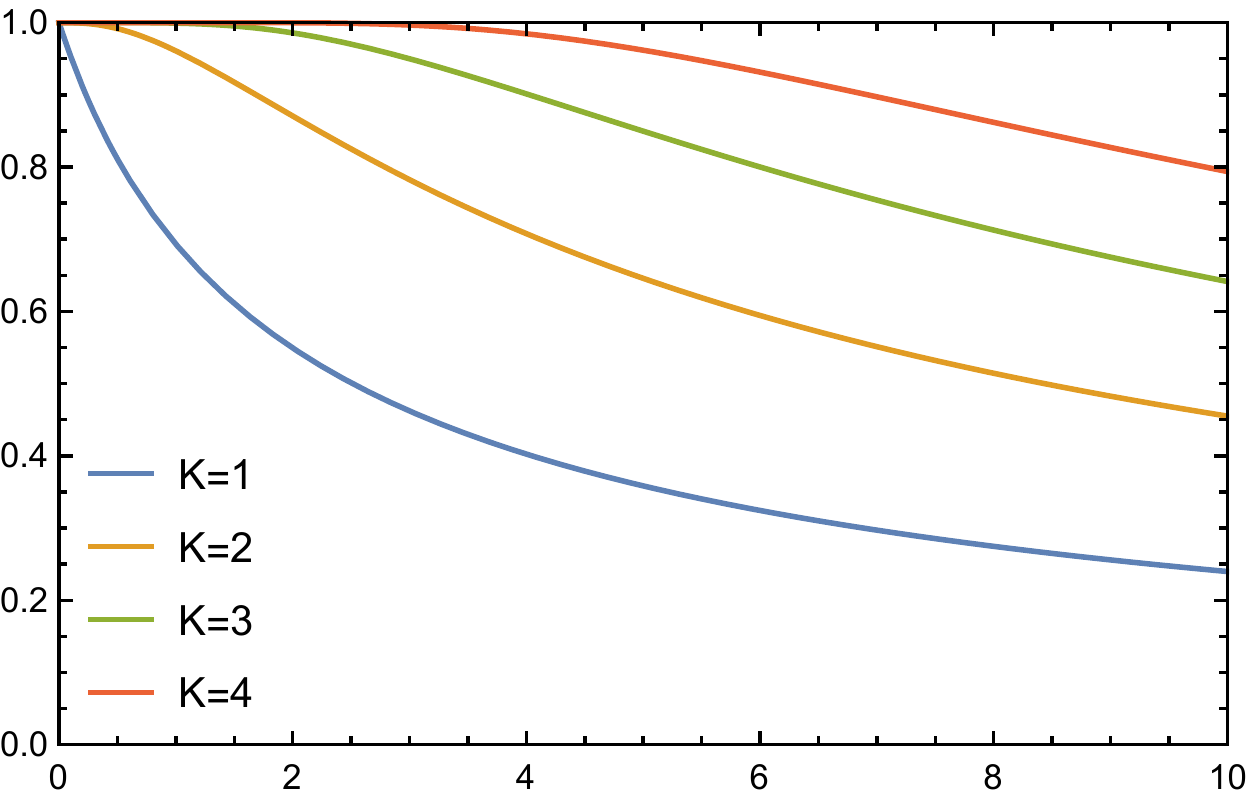}
  \caption{Jamming constant as a function of $c$\vspace{8pt}}
  \label{fig:SFAP-total}
\end{subfigure}
\caption{SFAP  model with $K=4$}
\end{figure}

As in the Tetris model, the proportion of nodes with the same frequency is a relevant quantity.
 We plot the jamming densities for the first four frequencies for increasing values of $c$ in Figure~\ref{fig:SFAP}. Observe that in this case the density decreases with the frequency.
  The total number of active nodes is given by the sum of the heights, as displayed in Figure~\ref{fig:SFAP-total}.

\begin{remark}
\normalfont
Observe that at each step the newly selected vertex $v$ is added to the set $\mathcal{A}_k$ if and only if $v$ has at least some connections to all the sets $\mathcal{A}_j$ with $1\leq j<k$, and has no connections with the set $\mathcal{A}_k$. 
Further, as in the previous cases, since the number of connections to the set $\mathcal{A}_j$, in the limit, is Poisson$(c\alpha_j)$ distributed, we obtain that this probability is given by the function $\delta_k$.
\end{remark}
\begin{remark} \normalfont In the random graphs literature the SFAP model is used as a greedy algorithm for finding an upper bound on the chromatic number of the Erd\H{o}s-R\'enyi random graph \cite{M84,PW97}. However, the SFAP version in this paper uses a fixed $K$, which is why Theorem~\ref{th:sfap} does not approximate the chromatic number, and gives the fraction of vertices that can be colored in a greedy manner with $K$ given colors instead. 
\end{remark}

\section{Proofs}\label{sec:proofs}
In this section we first prove Theorem~\ref{th:threshold} for the Threshold model. The proofs for Theorems~\ref{th:tetris} and \ref{th:sfap} use similar ideas except for the precise details, which is why we present these proofs in a more concise form. For the same reason, we give the proof of Proposition~\ref{prop:thres} and skip the proofs of the similar results in Proposition~\ref{prop:tetris} and Proposition~\ref{prop:sfap}.

\subsection{Proof of Theorem~\ref{th:threshold}}

\emph{Proof of Proposition~\ref{prop:thres}}
The difference between the Threshold model and Algorithm~\ref{algo-th} lies in the fact that the activation process in the Threshold model takes place on a given realization  of $G(n,p_n)$, whereas Algorithm~\ref{algo-th} generates the graph sequentially. To see that $(\mathcal{G}(n),\mathcal{A}(n))$ is indeed distributed as $(G(n,p_n),\mathcal{I}(n))$, it suffices to produce a coupling such that the graphs $G(n,p_n)$ and $\mathcal{G}(n)$ are identical and $\mathcal{I}(t)=\mathcal{A}(t)$ for all $1\leq t\leq n$.
For that purpose, associate an independent uniform$[0,1]$ random variable $U_{i,j}$ to each unordered pair $(i,j)$ both in the Threshold model and in Algorithm~\ref{algo-th}, for $1\leq i<j\leq n$. It can be seen that if we keep only those edges for which $U_{i,j}\leq p_n$, the resulting graph is distributed as $G(n,p_n)$. Therefore, when we create edges in both graphs according to the same random variables $U_{i,j}$, we ensure that $G(n,p_n)=\mathcal{G}(n).$

Now to select vertices uniformly at random from the set of all vertices that have not been selected yet, initially choose a random permutation of the set $\{1,2,\ldots, n\}$ and denote it by $\{\sigma_1,\sigma_2,\ldots,\sigma_n\}$. In both the Threshold model and Algorithm~\ref{algo-th}, at time $t$, select the vertex with index $\sigma_t$. Now, at time $t$, Algorithm~\ref{algo-th} only discovers the edges satisfying $U_{\sigma_t,j}\leq p_n$ for $j\in \mathcal{A}(t)$. Observe that this is enough for deciding whether $\sigma_t$ will be active or not. Therefore, if $\sigma_t$ becomes active in the Threshold model, then it will become active in Algorithm~\ref{algo-th} as well, and vice versa. We thus end up  getting precisely the same set of active vertices in the original model and the algorithm, which completes the proof.
\qed

We now proceed to prove Theorem~\ref{th:threshold}. The proof relies on a decomposition of the rescaled process as a sum of a martingale part and a drift part, and then showing that the martingale part converges to zero and the drift part converges to the appropriate limiting function.
 Let $\xi_k^n(t+1)$ be the number of edges created at step $t+1$ between the random vertex selected at step $t+1$  and the vertices in $\mathcal{A}_k(t)$. Also, for notational consistency, define $\xi_{-1}^n\equiv 0 $, and let $\xi^n(t+1):=\sum_{k=0}^{K-1}\xi_k^n(t+1)$. Recall that an empty sum is taken to be zero. Note that, for any $0\leq k\leq K-1$,
\begin{equation}\label{defn:A-k-n}
 \begin{split}
  A_k^n(t+1)= A_k^n(t)+\zeta_{ k}^n(t+1),
 \end{split}
\end{equation}where
\begin{equation}
 \zeta_{ k}^n(t+1)=\xi_{k-1}^n(t+1)-\xi_k^n(t+1)+ \ind{\xi^{n}(t+1)=k}
\end{equation}
if $\xi^n(t+1)\leq K-1$ and $\xi^n_{K-1}(t+1)=0$, and $\zeta_{ k}^n(t+1)=0$ otherwise.
To see this, observe that at time $t+1$, if the number of new connections to the set of active vertices exceeds $K-1$, or a connection is made to some active vertex that already has $K-1$ active neighbors, then the newly selected vertex cannot become active. Otherwise, the newly selected vertex instantly becomes active, and if the total number of new connections to $\mathcal{A}^n(t)$ is $j$ for some $j\leq K-1$, then $\xi^n_{k}(t+1)$ vertices of $\mathcal{A}_k(t)$ will now have $k+1$ active neighbors, for $0\leq k\leq K-2$, and the newly active vertex will be added to $\mathcal{A}_j(t+1)$.

Observe that $\{\bld{A}^n(t)\}_{t\geq 0}=\{(A_0^n(t),\dots,A_{\sss K-1}^n(t))\}_{t\geq0}$ is an $\R^K$-valued Markov process. Moreover, for any $0\leq k\leq K-1$, given the value of $\bld{A}^n(t)$,
$\xi_k^n(t+1)\sim \mathrm{Bin}(A_k^n(t),p_n)$ and $\xi_0^n(t+1),\dots,\xi_{\sss K-1}^n(t+1)$ are mutually independent when conditioned on $\bld{A}^n(t)$. Write $A_{\leq r}^n(t)= A_1^n(t)+\dots+A_r^n(t)$.
For a random variable $X\sim\mathrm{Bin}(n,p)$, denote $\mathsf{B}(n,p;k)=\prob{X\leq k}$ and $\mathsf{b}(n,p;k)=\prob{X= k}$.
Now we need the following technical lemma:
\begin{lemma}\label{lem:bin-condn-expt}
 Let $X_1,\dots,X_r$ be $r$ independent random variables with $X_i$ distributed as $\mathrm{Bin}(n_i,p)$. Then, for any $1\leq R\leq \sum_{i=1}^rn_i$,
 \begin{subequations}
 \begin{equation}
  \expt{X_i|X_1+\dots+X_r\leq R}=n_ip\frac{\prob{Z_1\leq R-1}}{\prob{Z_2\leq R}}
 \end{equation}
 and
 \begin{equation}
  \expt{X_i(X_i-1)|X_1+\dots+X_r\leq R}\leq \frac{n_i(n_i-1)p^2}{\prob{Z_2\leq R}},
 \end{equation}
 \end{subequations}
 where $Z_1\sim\mathrm{Bin}\left(\sum_{i=1}^rn_i-1,p\right)$ and $Z_2\sim\mathrm{Bin}\left(\sum_{i=1}^rn_i,p\right)$.
\end{lemma}
\begin{proof}
 Note that $\expt{X_i|X_1+\dots+X_r=j}=n_ij/(n_1+\dots+n_r)$. Therefore,
 \begin{equation}
  \begin{split}
   \expt{X_i}&=\expt{X_i|X_1+\dots+X_r\leq R}\prob{X_1+\dots+X_r\leq R}\\
   &\hspace{.3cm}+\sum_{j=R+1}^{n_1+\dots+n_r}\expt{X_i|X_1+\dots+X_r=j}\prob{X_1+\dots+X_r=j}.
  \end{split}
 \end{equation}Thus, since 
 \begin{equation}
 \frac{j}{p\sum_{i=1}^rn_i}\mathrm{b}\left(\sum_{i=1}^rn_i,p;j\right)= \mathrm{b}\left(\sum_{i=1}^rn_i-1,p;j-1\right)
 \end{equation}
 we get
 \begin{equation}
  \begin{split}
   &\expt{X_i|X_1+\dots+X_r\leq R}\\
   =&\frac{n_ip}{\prob{X_1+\dots+X_r\leq R}}\bigg(1-\frac{1}{p\sum_{i=1}^rn_i}\sum_{j=R+1}^{n_1+\dots+n_r}j \prob{X_1+\dots+X_r=j}\bigg)\\
   =&\frac{n_ip}{\prob{Z_2\leq R}}(1-\prob{Z_1\geq R})=n_ip\frac{\prob{Z_1\leq R-1}}{\prob{Z_2\leq R}}.
  \end{split}
 \end{equation}
Further,
 \begin{equation}
 \begin{split}
  n_i(n_i-1)p^2&=\expt{X_i(X_i-1)}\\
  &\geq \expt{X_i(X_i-1)|X_1+\dots+X_r\leq R}\prob{Z_2\leq R}
  \end{split}
 \end{equation}
 and the proof is complete.
 \qed
\end{proof}
Using Lemma~\ref{lem:bin-condn-expt} we get the following expected values:
\begin{equation}
 \begin{split}
  &\expt{\xi_k^n(t+1)\ind{\xi^n(t+1) \leq K-1, \xi^n_{\sss K-1}(t+1)=0}\big\vert \bld{A}^n(t)}\\
  &=A_k^n(t)p_n(1-p_n)^{A_{K-1}^n(t)}\mathsf{B}(A_{\sss \leq K-2}^n(t)-1,p_n; K-2),
 \end{split}
\end{equation}
and thus, for $0\leq k\leq K-1$,
\begin{equation}
\begin{split}
&\expt{\zeta^n_k(t+1)|\bld{A}^n(t)}\\
&=(A_{k-1}^n(t)-A_k^n(t))p_n(1-p_n)^{A_{K-1}^n(t)}\mathsf{B}(A_{\sss \leq K-2}^n(t)-1,p_n; K-2)\\
&\hspace{1cm}+\mathsf{b}(A^n_{\sss \leq K-2},p_n;k)(1-p_n)^{A^n_{\sss K-1}(t)},
\end{split}
\end{equation}
where $A_{-1}^n\equiv A_K^n\equiv 0$.
For $\boldsymbol{i}=(i_0,\dots,i_{\sss K-1})\in\{1,\ldots,n\}^K$, define the drift function
\begin{equation}
\Delta_k^n(\bld{i}):=\expt{\zeta^n_k(t+1)\given \bld{A}^n(t)=\bld{i}}.
\end{equation}
 Denote $\delta_k^n(\boldsymbol{\alpha}):=\Delta_k^n(n\boldsymbol{\alpha})$ for $\bld{\alpha}\in [0,1]^{K}$, and 
 $$\bld{\delta}^n(\bld{\alpha}):=(\delta_0^n(\boldsymbol{\alpha}),\ldots,\delta_{K-1}^n(\boldsymbol{\alpha})).$$
 Recall the definition of $\bld{\delta}(\bld{\alpha})=(\delta_0(\bld{\alpha}),\dots, \delta_{K-1}(\bld{\alpha}))$ in \eqref{def:delta-th}.
\begin{lemma}[Convergence of the drift function]\label{lem:lips-cont}
The time-scaled drift function $\bld{\delta}^n$ converges uniformly on $[0,1]^K$ to the Lipschitz-continuous function $\bld{\delta}:[0,1]^K\mapsto [0,1]^K$.
\end{lemma}
\begin{proof}
Observe that $\bld{\delta}(\cdot)$ is continuously differentiable, defined on a compact set, and hence, is Lipschitz continuous. Also, $\bld{\delta}^n$ converges to $\bld{\delta}$ point-wise and the uniform convergence is a consequence of the continuity of $\bld{\delta}$ and the compactness of the support.
\qed
\end{proof}

Recall that $A_k^n(0)=0$ for $0\leq k\leq K-1$. The Doob-Meyer decomposition of \eqref{defn:A-k-n} gives
\begin{equation}
\begin{split}
 A_k^n(t)&=\sum_{i=1}^t\zeta_k^n(i)=M_k^n(t)+\sum_{i=1}^t\expt{\zeta_k^n(i)|\bld{A}^n(i-1)},
 \end{split}
\end{equation}where $(M_k^n(t))_{t\geq 1}$ is a locally square-integrable martingale. We can write
\begin{equation}\label{eq:mart-decompose}
 \begin{split}
  \alpha_k^n(t)&=\frac{M_k^n(\floor{nt})}{n}+\frac{1}{n}\sum_{i=1}^{\floor{nt}}\Delta_k^n(\bld{A}^n(i-1))\\
  &=\frac{M_k^n(\floor{nt})}{n}+\frac{1}{n}\int_0^{\floor{nt}-1}\Delta_k^n(\bld{A}^n(s))\dif s\\
  &=\frac{M_k^n(\floor{nt})}{n}+\int_0^{t}\Delta_k^n(\bld{A}^n(ns))\dif s-\int_{(\floor{nt}-1)/n}^t\Delta_k^n(\bld{A}^n(ns))\dif s\\
  &= \frac{M_k^n(\floor{nt})}{n}+\int_0^{t}\delta_k^n(\bld{\alpha}^n(s))\dif s-\int_{(\floor{nt}-1)/n}^t\delta_k^n(\bld{\alpha}^n(s))\dif s.
 \end{split}
\end{equation}
First we show that the martingale terms converge to zero. \begin{lemma}\label{lem:conv-mart}For all $0\leq k\leq K-1$, as $n\to\infty$
\begin{equation}\label{eq:mart-conv}
\sup_{s\in [0,1]}\frac{|M_k^n(ns)|}{n}\dto 0.
\end{equation}
\end{lemma}
\begin{proof}
The scaled quadratic variation term can be written as
 \begin{equation}\label{expr:qv}
  \begin{split}
  \frac{1}{n^2}\langle M_k^n \rangle(\floor{ns})&=\frac{1}{n^2}\sum_{i=1}^{\floor{ns}}\var{\zeta_k^n(i)|\bld{A}^n(i-1)}.
  \end{split}
 \end{equation} Now, using Lemma~\ref{lem:bin-condn-expt} we get,
 \begin{equation}
 \begin{split}
  &\expt{\xi_k^n(i)(\xi_k^n(i)-1)\ind{\xi^n(t+1) \leq K-1, \xi^n_{\sss K-1}(t+1)=0}\big\vert\bld{A}^n(i-1)}\\
  &\leq  A_k^n(i-1)(A_k^n(i-1)-1)p_n^2(1-p_n)^{A_{\sss K-1}^n(i-1)}\leq c^2
  \end{split}
 \end{equation} for all large enough $n$, where we have used that $A_k^n(i-1)\leq n$ and $p_n=c/n$ in the last step. Thus, there exists a constant $C>0$ such that for all large enough $n$,
 \begin{equation}
  \expt{\zeta_k^n(i)|\bld{A}^n(i-1)}\leq C.
\end{equation}  Therefore, \eqref{expr:qv} implies $
  \langle M_k^n\rangle/n^2\dto 0$ and this proves \eqref{eq:mart-conv}.
  \qed
\end{proof} 
Also, Lemma~\ref{lem:lips-cont} implies that $\sup_{n\geq 1}\sup_{\bld{x}\in [0,1]^{K}}|\delta_k^n(\bld{x})|<\infty $ for any $0\leq k\leq K-1$. Therefore,
 \begin{equation}
  \int_{(\floor{nt}-1)/n}^t\delta_k^n(\bld{\alpha}^n(s))\dif s\leq \varepsilon_n',
 \end{equation}
 where $\varepsilon_n'$ is non-random, independent of $t,k$ and $\varepsilon_n'\to 0$. Thus, for any $t\in [0,1]$,
\begin{equation}\label{eq:alpha-diff}
\begin{split}
 &\sup_{s\in [0,t]}\norm{\bld{\alpha}^n(s)-\bld{\alpha}(s)}\\
 &\leq \sup_{s\in [0,t]}\frac{\norm{\bld{M}_n(ns)}}{n}+\int_0^t\sup_{u\in [0,s]}\norm{\bld{\delta}^n(\bld{\alpha}^n(u))-\bld{\delta}(\bld{\alpha}(u))}\dif s +\varepsilon'_n.
 \end{split}
\end{equation} Now, since $\bld{\delta}$ is a Lipschitz-continuous function, there exists a constant $C>0$ such that $\norm{\bld{\delta}(\bld{x})-\bld{\delta}(\bld{y})}\leq C \norm{\bld{x}-\bld{y}}$ for all $\bld{x},\bld{y}\in [0,1]^K$. Therefore,
\begin{equation}
 \begin{split}\label{eq:delta-alpha-diff}
  \sup_{u\in [0,s]}\norm{\bld{\delta}^n(\bld{\alpha}^n(u))-\bld{\delta}(\bld{\alpha}(u))}\leq \sup_{\bld{x}\in [0,1]^K}&\norm{\bld{\delta}^n(\bld{x})-\bld{\delta}(\bld{x})}\\
  & + C\sup_{u\in [0,s]} \norm{\bld{\alpha}^n(u)-\bld{\alpha}(u)}.
 \end{split}
\end{equation}
Lemma~\ref{lem:lips-cont}, \eqref{eq:alpha-diff} and \eqref{eq:delta-alpha-diff} together imply that
\begin{equation}
 \begin{split}
  \sup_{s\in [0,t]}\norm{\bld{\alpha}^n(s)-\bld{\alpha}(s)}&\leq  C\int_0^t\sup_{u\in [0,s]} \norm{\bld{\alpha}^n(u)-\bld{\alpha}(u)}\dif s+ \varepsilon_n,
 \end{split}
\end{equation}
where $\varepsilon_n\dto 0$.
Using Gr\H{o}nwall's inequality \cite[Proposition 6.1.4]{ODE_martin}, we get
\begin{equation}
 \sup_{s\in [0,t]}\norm{\bld{\alpha}^n(s)-\bld{\alpha}(s)}\leq  \varepsilon_n \e^{Ct}.
\end{equation}
Thus the proof of Theorem~\ref{th:threshold} is complete.
\qed

\subsection{Proof of Theorem~\ref{th:tetris}} The proof of Theorem~\ref{th:tetris} is similar to the proof of Theorem~\ref{th:threshold}. Again denote $A^n_k(t)=|\mathcal{A}_k(t)|$, where $\mathcal{A}_k(t)$ is the number of active vertices at height $k$ at time $t$. Note here that $A^n_0(t)$  is the set of frozen vertices.
Let $\xi_k^n(t+1)$ be the number of vertices in $\mathcal{A}_k(t)$ that are paired to the vertex selected at time $t+1$ by Algorithm~\ref{algo-tetris}.
Then, for $1\leq k\leq K$,
 \begin{equation}
  A_k^n(t+1)=A_k^n(t)+\zeta_k^n(t+1),
 \end{equation}
 where, for $k\geq 2$,
 \begin{equation}
 \begin{split}
  \zeta_k^n(t+1)&=
   \begin{cases}
   1 \quad \text{if } \xi_r^n(t+1)=0, \ \forall r\geq k, \ \xi_{k-1}^n(t+1)>0,\\
   0 \quad  \text{otherwise},
   \end{cases}\\
   \zeta_1^n(t+1)&=
   \begin{cases}1 \quad \text{if } \xi_r^n(t+1)=0, \ \forall r\geq 1,\\
   0 \quad  \text{otherwise},
   \end{cases}
   \end{split}
 \end{equation}
Indeed, observe that if $j$ is the maximum index for which the new vertex selected at time $t+1$ makes a connection to $A_j^n(t)$,  and $j\leq K-1$, then $v$ will be assigned height $j+1$.
Therefore,
\begin{equation}
\begin{split}
 &\expt{\zeta_k^n(t+1)|\bld{A}^n(t)}
 \\&=
 \begin{cases}
  \big(1-(1-p_n)^{A_{k-1}^n(t)}\big)(1-p_n)^{A_k^n(t)+\dots+A_{\sss K}^n(t)}, & \text{for } k\geq 2,\\
  (1-p_n)^{A_1^n(t)+\dots+A_K^n(t)},& \text{for } k=1.
 \end{cases}
 \end{split}
\end{equation}
For $\boldsymbol{i}=(i_1,\dots,i_{\sss K})\in[n]^{K}$, define the drift rate functions
\begin{equation}
 \Delta_k^n(\bld{i})=\Delta_k^n(\bld{i}):=\expt{\zeta^n_k(t+1)\given \bld{A}^n(t)=\bld{i}},
\end{equation}
and denote  $\delta_k^n(\boldsymbol{\alpha})=\Delta_k^n(n\boldsymbol{\alpha})$ for $\bld{\alpha}\in [0,1]^{K}$, $\bld{\delta}^n(\bld{\alpha})=(\delta_1^n(\boldsymbol{\alpha}),\dots,\delta_{K}^n(\boldsymbol{\alpha}))$.
Also, let $\bld{\delta}(\bld{\alpha})=(\delta_1(\bld{\alpha}),\dots, \delta_{K}(\bld{\alpha}))$ where we recall the definition of $\delta_k(\cdot)$ from
\eqref{def:delta-tetris}.
\begin{lemma}[Convergence of the drift function]\label{lem:lips-cont-t}
The time-scaled drift function $\bld{\delta}^n$ converges uniformly on $[0,1]^{K}$ to the Lipschitz continuous function $\bld{\delta}:[0,1]^{K}\mapsto [0,1]$.
\end{lemma}

The above lemma can be seen from the same arguments as used in the proof of Lemma~\ref{lem:lips-cont}. In this case also, we can obtain a martingale decomposition similar to \eqref{eq:mart-decompose}. Here, the increments $\zeta_k^n(\cdot)$ values are at most 1. Therefore, the quadratic variation of the scaled martingale term is at most $1/n$. Hence one obtains the counterpart of Lemma~\ref{lem:conv-mart} in this case, and the proof can be completed using similar arguments as in the proof of Theorem~\ref{th:tetris}.
\qed

\subsection{Proof of Theorem~\ref{th:sfap}}

As in the previous section, we only compute the drift function and the rest of the proof is similar to the proof of Theorem~\ref{th:threshold}.
Let $A^n_k(t)=|\mathcal{A}_k(t)|$, where $\mathcal{A}_k(t)$ is obtained from Algorithm~\ref{algo-sfap}, and let
 $\xi_k^n(t+1)$ be the number of vertices of $\mathcal{A}_k(t)$ that are paired to the vertex selected randomly among the set of unexplored vertices at time $t+1$.
Then,
 \begin{equation}
  A_k^n(t+1)=A_k^n(t)+\zeta_k^n(t+1),
 \end{equation}
 where, for any $1\leq k\leq K$,
 \begin{equation}
  \zeta_k^n(t+1)=
   \begin{cases}1 \quad \text{if } \xi_r^n(t+1)>0, \ \forall\ 1\leq r< k, \text{ and } \xi_{k}^n(t+1)=0,\\
   0 \quad  \text{otherwise}.
   \end{cases}
 \end{equation}
This follows by observing that the new vertex selected at time $t+1$ is assigned frequency $j$, for some $j\leq K$, if and only if the new vertex makes no connection with $\mathcal{A}^n_j(t)$, and has at least one connection with $\mathcal{A}_k^n(t)$ for all $1\leq k\leq j-1$. Hence the respective expectations can be written as
\begin{equation}
 \expt{\zeta_k^n(t+1)|\bld{A}^n(t)}=(1-p_n)^{A_{k}^n(t)}\prod_{r=1}^{k-1}\left(1-(1-p_n)^{A_r^n(t)}\right),
\end{equation}
for $1\leq k\leq K$.
  Defining the functions $\Delta$, $\delta$ suitably, as in the proof of the Tetris model, the current proof can be completed in the exact same manner.
\qed

\section{Further research}\label{sec:fu}
 This paper considers Random Sequential Adsorption (RSA) on the Erd\H{o}s-R\'enyi random graph and relaxes the strict hardcore interaction between active nodes in three different ways, leading to the Threshold model, the Tetris model and the SFAP model. The Threshold model constructs a greedy maximal $K$-independent set. For $K=1$ it is known that the size of the \emph{maximum} set is almost twice as large as the size of a greedy \emph{maximal} set \cite{F90,CE15}. From the combinatorial perspective, it is interesting to study the size of the maximum $K$-independent set in random graphs in order to quantify the gap with the greedy solutions. Similarly, in the context of the SFAP model, it is interesting to find the maximum fraction of vertices that can be activated if there are $K$ different frequencies. Another fruitful direction is to determine the jamming constant for the three generalized RSA models when applied to other classes of random graphs, such as random regular graphs or random graphs with more general degree distributions such as inhomogeneous random graphs, the configuration model or preferential attachment graphs.

{\small
\section*{Acknowledgments}
This research was financially supported by an ERC Starting Grant and by The Netherlands Organization for Scientific Research (NWO) through TOP-GO grant 613.001.012 and Gravitation Networks grant 024.002.003.
We thank Sem Borst, Remco van der Hofstad, Thomas  Meyfroyt, and  Jaron Sanders for a careful reading of the manuscript.

\bibliographystyle{abbrv}

\bibliography{jam}

}
\end{document}